\documentclass[12pt]{article}
\usepackage{amsmath,enumerate,amsfonts,amssymb,color,graphicx,amsthm}
\usepackage{multirow,caption,subcaption}
\usepackage{hyperref}
\usepackage{todonotes}
\usepackage[normalem]{ulem}
\usepackage{cite}

\numberwithin{equation}{section}

\def\ep{{\epsilon}}

\def\p{{\partial}}

\newcounter{marnote}

\begin{document}
\newtheorem{lem}{Lemma}[section]
\newtheorem{rem}{Remark}
\newtheorem{question}{Question}
\newtheorem{prop}{Proposition}
\newtheorem{cor}{Corollary}
\newtheorem{thm}{Theorem}[section]
\newtheorem{definition}{Definition}[section]
\newtheorem{openproblem}{Open Problem}
\newtheorem{conjecture}{Conjecture}

\newenvironment{dedication}
  {
%\clearpage           % we want a new page
%   \thispagestyle{empty}% no header and footer
   \vspace*{\stretch{.2}}% some space at the top
   \itshape             % the text is in italics
%   \raggedleft          % flush to the right margin
  }
  {
%\par % end the paragraph
   \vspace{\stretch{1}} % space at bottom is three times that at the top
%   \clearpage           % finish off the page
  }

\title
{Existence results for Toda systems with sign-changing prescribed functions: Part I
\thanks{%Partially supported by the National Science Foundation of China (Grant No. 11401575).
Email: sunll\underline{~}m@163.com (Linlin Sun); zhuxiaobao@ruc.edu.cn (Xiaobao Zhu)}}
\author{Linlin Sun\\
School of Mathematics and Computational Science\\
Xiangtan University\\
Xiangtan 411105, P. R. China\\
\\
Xiaobao Zhu\\
School of Mathematics\\
Renmin University of China\\
Beijing 100872, P. R. China\\
}

\date{ }
\maketitle

%\centerline{ \it
%%\begin{dedication}
%Dedicated to Professor Jiayu Li on the occasion of his 60th birthday
%%\end{dedication}
%}

\begin{abstract}
  Let $(M,g)$ be a compact Riemann surface with area $1$, we shall study the Toda system
 \begin{align}\label{eq-jlw}
 \begin{cases}
 -\Delta u_1 = 2\rho_1(h_1e^{u_1}-1) - \rho_2(h_2e^{u_2}-1),\\
 -\Delta u_2 = 2\rho_2(h_2e^{u_2}-1) - \rho_1(h_1e^{u_1}-1),
 \end{cases}
 \end{align}
 on $(M,g)$ with $\rho_1=4\pi$, $\rho_2\in(0,4\pi)$, $h_1$ and $h_2$ are two smooth functions on $M$.
 In Jost-Lin-Wang's celebrated article (Comm. Pure Appl. Math., 59 (2006), no. 4, 526--558), they obtained
 a sufficient condition for the existence of Eq. \eqref{eq-jlw} when $h_1$ and $h_2$
 are both positive. In this paper, we shall improve this result to the case $h_1$ and $h_2$ can change signs.
 We shall pursue a variational method and use the standard blowup analysis. Among other things, the main contribution in our proof is to show the blowup can
 only happen at one point where $h_1$ is positive.
\end{abstract}

\setcounter {section} {0}

\section{Introduction}

Let $(M,g)$ be a compact Riemann surface with area $1$, $h_i\in C^\infty(M)$ and $\rho_i$ be positive constant for $i=1,2$.
The critical point $(u_1,u_2)$ of the functional
\begin{align*}%\label{functional-toda}
J_{\rho_1,\rho_2}(u_1,u_2) = \frac{1}{3}\int_M \left(|\nabla u_1|^2+\nabla u_1\nabla u_2+|\nabla u_2|^2\right)
+ \rho_1\int_M u_1 + \rho_2\int_M u_2
\end{align*}
on the Hilbert space
\begin{align*}%\label{space}
\mathcal{H} = \left\{(u_1,u_2)\in H^1(M)\times H^1(M):~\int_M h_1e^{u_1}=\int_M h_2e^{u_2}=1\right\}
\end{align*}
satisfies
\begin{align}\label{eq-toda-1}
 \begin{cases}
 -\Delta u_1 = 2\rho_1(h_1e^{u_1}-1) - \rho_2(h_2e^{u_2}-1),\\
 -\Delta u_2 = 2\rho_2(h_2e^{u_2}-1) - \rho_1(h_1e^{u_1}-1).
 \end{cases}
\end{align}
In the literal, people calls \eqref{eq-toda-1} as Toda system. It can be seen as the Frenet frame of holomorphic curves in $\mathbb{CP}^2$ (see \cite{G97})
in geometry, and also arises in physics in the study of the nonabelian Chern-Simons theory in the self-dual case,
when a scalar Higgs field is coupled to a gauge potential; see \cite{Dun95,Ya01,Tar08}.
One can easily find that Toda system \eqref{eq-toda-1} is a generalization of the mean field equation
\begin{align}\label{eq-mfe}
-\Delta u = \rho(he^u-1).
\end{align}
If $u$ is a solution of Eq. \eqref{eq-mfe}, then one has $\int_Mhe^u=1$. Therefore, people solves Eq. \eqref{eq-mfe} in Hilbert space
$$X=\left\{u\in H^1(M):\int_M he^u=1\right\}.$$
Since Eq. \eqref{eq-mfe} has a variational structure, thanks to the Moser-Trudinger inequality (c.f. \cite{Fon,DJLW97})
\begin{align*}%\label{ineq-mt}
\log\int_M e^u\leq \frac{1}{16\pi}\int_M|\nabla u|^2+\int_M u+C,
\end{align*}
it has a minimal type solution in $X$ when $\rho\in(0,8\pi)$. However, when $\rho=8\pi$, the situation becomes subtle. The famous Kazdan-Warner problem \cite{KW74} states that
under what kind of condition on $h$, the equation
\begin{align}\label{eq-8pi}
-\Delta u = 8\pi(he^u-1)
\end{align}
has a solution. Necessarily, one needs $\max_Mh>0$. By using blowup argument and a variational method, Ding, Jost, Li and Wang \cite{DJLW97} attacked this problem successfully.
Assuming $h$ is positive, if
\begin{align}\label{cond-djlw}
\Delta\log h(p_0)+8\pi-2K(p_0)>0,
\end{align}
where $K$ is the Gauss curvature of $M$, $p_0$ is the maximum point of $2\log h(p)+A_p$ on $M$,
$A_p=\lim_{x\to p}\left(G_p(x)+4\log \text{dist}(x,p)\right)$ and $G_p$ is the Green function which satisfies
\begin{align*}
\begin{cases}
-\Delta G_p = 8\pi(\delta_p-1),\\
\int_MG_p=0,
\end{cases}
\end{align*}
then Eq. \eqref{eq-8pi} has a minimal type solution. Yang and the second author \cite{YZ17} generalized this existence result to the case $h$ is nonnegative.
With different arguments, the first author and Zhu \cite{SZ24} and the second author \cite{Z24} proved respectively the Ding-Jost-Li-Wang condition \eqref{cond-djlw}
is still sufficient for the existence of Eq. \eqref{eq-8pi} when $h$ changes signs. The mentioned works are all based on variational method. We remark that these results were
also obtained by using flow method \cite{LZ19,SZ21,WY22,LX22}.

To well understand the analytic properties of the Toda system,
Jost-Wang \cite{JW01} derived the Moser-Trudinger inequality for it:
\begin{align}\label{ineq-toda}
\inf_{(u_1,u_2)\in\mathcal{H}} J_{\rho_1,\rho_2} \geq -C~~~~\text{iff}~~~~\rho_1,\rho_2\in(0,4\pi].
\end{align}
From this inequality, we know that $J_{\rho_1,\rho_2}$ is coercive and hence attains its infimum when $\rho_1,\rho_2\in(0,4\pi)$. However,
when $\rho_1$ or $\rho_2$ equals $4\pi$, the existence problem also becomes subtle. In this paper, we shall put our attention on minimal type solution.
Hence, throughout this paper, we assume $\rho_i\leq4\pi$, $i=1,2$.

Let us review the existence result when one of $\rho_1$ and $\rho_2$ equals $4\pi$, which was obtained by Jost, Lin and Wang when $h_1$ and $h_2$ are both positive.

\begin{thm}[Jost-Lin-Wang\cite{JLW06}]
Let $(M,g)$ be a compact Riemann surface with Gauss curvature $K$. Let $h_1,h_2\in C^2(M)$ be two positive functions
 and $\rho_2\in(0,4\pi)$. Suppose that
\begin{align}\label{JLW-cond-1}
\Delta\log h_1(x)+(8\pi-\rho_2)-2K(x)>0,~~~\forall x\in M,
\end{align}
then $J_{4\pi,\rho_2}$ has a minimizer $(u_1,u_2)\in\mathcal{H}$ which satisfies
\begin{align}\label{eq-jlw-1}
 \begin{cases}
 -\Delta u_1 = 8\pi(h_1e^{u_1}-1) - \rho_2(h_2e^{u_2}-1),\\
 -\Delta u_2 = 2\rho_2(h_2e^{u_2}-1) - 4\pi(h_1e^{u_1}-1).
 \end{cases}
\end{align}
\end{thm}

When $\rho_1=\rho_2=4\pi$ and both $h_1$ and $h_2$ are positive, we have
\begin{thm}[Li-Li\cite{LL05}, Jost-Lin-Wang\cite{JLW06}]\label{thm-critical-0}
Let $(M,g)$ be a compact Riemann surface with Gauss curvature $K$. Let $h_1,h_2\in C^2(M)$ be two positive functions. Suppose that
\begin{align}\label{JLW-cond-2}
\min\{\Delta\log h_1(x),\Delta\log h_2(x)\}+4\pi-2K(x)>0,~~~\forall x\in M,
\end{align}
then $J_{4\pi,4\pi}$ has a minimizer $(u_1,u_2)\in\mathcal{H}$ which satisfies
\begin{align*}%\label{eq-jlw-2}
\begin{cases}
-\Delta u_1 = 8\pi(h_1e^{u_1}-1) - 4\pi(h_2e^{u_2}-1),\\
-\Delta u_2 = 8\pi(h_2e^{u_2}-1) - 4\pi(h_1e^{u_1}-1).
\end{cases}
\end{align*}
\end{thm}

We remark that Li-Li obtained Theorem \ref{thm-critical-0} when $h_1=h_2=1$ and Jost-Lin-Wang obtained it for general positive $h_1$ and $h_2$.

Motivated mostly by works in \cite{DJLW97, YZ17, SZ24,Z24}, we would like to release conditions \eqref{JLW-cond-1} and \eqref{JLW-cond-2} as much as possible.
%To describe our result, we first recall the Green function on $(M,g)$: Let $G_y$ be the Green function which satisfies
%\begin{align}\label{eq-green}
%\begin{cases}
%&G_p(x) = 1-\delta_p(x),~~x\in M,\\
%&\int_M G_p = 0.
%\end{cases}
%\end{align}
%In a local isothermal coordinates around $p$, we have
%\begin{align}\label{green-expre}
%G_p(x) =& -\frac{1}{2\pi}\log r + A_p + b_1(p)r\cos\theta + b_2(p)r\sin\theta\nonumber\\
%&+ (r\cos\theta, r\sin\theta)\left(
%                              \begin{array}{cc}
%                                c_{11}(p) & c_{12}(p) \\
%                                c_{21}(p) & c_{22}(p) \\
%                              \end{array}
%                            \right)\left(
%                                     \begin{array}{c}
%                                       r\cos\theta \\
%                                       r\sin\theta \\
%                                     \end{array}
%                                   \right) + O(r^3),
%\end{align}
%where $r=r(x)=\text{dist}_g(x,p)$.
Comparing with the sufficient conditions in \cite{DJLW97,YZ17,SZ24,Z24}, we believe that conditions \eqref{JLW-cond-1} and \eqref{JLW-cond-2} can release to
$h_i$ may change signs and the conditions only need hold on maximum points of the prescribed functions, namely $h_1$ and $h_2$. In the first step to
this aim, we are successful to release \eqref{JLW-cond-1} when $h_1$ and $h_2$ can change signs. Precisely, we have
\begin{thm}\label{thm-zhu-1}
Let $(M,g)$ be a compact Riemann surface with Gauss curvature $K$. Let $h_1, h_2\in C^2(M)$ which are positive somewhere and $\rho_2\in(0,4\pi)$. Denote $M_+=\{x\in M: h_1(x)>0\}$. Suppose that
 $$2\log h_1(p)+ A_1(p)=\max_{M_+}(2\log h_1(x)+A_1(x)),$$
where $A_1(p)$ is defined in \eqref{green-expression}. If
\begin{align}\label{zhu-cond-1}
\Delta\log h_1(p)+(8\pi-\rho_2)-2K(p)>0,
\end{align}
then $J_{4\pi,\rho_2}$ has a minimizer $(u_1,u_2)\in\mathcal{H}$ which satisfies \eqref{eq-jlw-1}.
\end{thm}

At the end of the introduction, we would like to mention some related works which deal with sign-changing potential in the critical case with respect to
Moser-Trudinger type inequalities (\cite{Mart09, YuZ24+}).
For the generalization of Theorem \ref{thm-critical-0}, we can also release the condition and
we shall give the details in a forthcoming paper.

The outline of the rest of the paper is following: In Sect. 2, we do some analysis on the minimizing sequence; In Sect. 3, we estimate the lower bound for
$J_{4\pi,\rho_2}$; Finally, we complete the proof of Theorem \ref{thm-zhu-1} in the last section. Throughout the whole paper,
the constant $C$ is varying from line to line and even in the same line, we do not distinguish
sequence and its subsequences since we just consider the existence result.

\section{Analysis on the minimizing sequence}
To show the functional $J_{4\pi,\rho_2}$ is bounded from below, we consider the perturbed functional $J_{4\pi-\ep,\rho_2}$. Since the infimum
of the functional $J_{4\pi-\ep,\rho_2}$ in $\mathcal{H}$ can be attained by $(u_1^\ep,u_2^\ep)$, we call $(u_1^\ep,u_2^\ep)$ the minimizing sequence and analysis
it in this section.

For $\rho_2\in(0,4\pi)$, in view of inequality \eqref{ineq-toda}, one knows for any $\ep\in(0,4\pi)$ there exists a $(u^\ep_1,u^{\ep}_2)\in\mathcal{H}$ such that
$J_{4\pi-\ep,\rho_2}(u_1^\ep,u_2^\ep)=\inf_{(u_1,u_2)\in\mathcal{H}} J_{4\pi-\ep,\rho_2}(u_1,u_2)$. Direct calculation shows on $M$,
\begin{align}\label{eq-uep}
\begin{cases}
-\Delta u^\ep_1=(8\pi-2\ep)(h_1e^{u_1^\ep}-1)-\rho_2(h_2e^{u^\ep_2}-1),\\
-\Delta u^\ep_2=2\rho_2(h_2e^{u^\ep_2}-1)-(4\pi-\ep)(h_1e^{u_1^\ep}-1).
\end{cases}
\end{align}

Denote $\overline{u^\ep_i}=\int_M u^\ep_i$ and $m^\ep_i=\max_{M}u^\ep_i=u_i^\ep(x_i^\ep)$ for some $x_i^{\ep}\in M$. %Assume $x_i^\ep\to p_i$ as $\ep\to0$.
Since $(u_1^\ep,u_2^\ep)$ minimizes $J_{4\pi-\ep,\rho_2}$ in $\mathcal{H}$, we have $\int_Me^{u_i^{\ep}}$ ($i=1,2$) is bounded from
below and above by two positive constants. Namely,
\begin{lem}\label{lem-bound}
There exist two positive constants $C_1$ and $C_2$ such that
\begin{align*}
C_1\leq \int_Me^{u_i^{\ep}} \leq C_2,~~i=1,2.
\end{align*}
\end{lem}
\begin{proof}
For $i=1,2$, the lower bound is easy since $\int_Mh_ie^{u_i^\ep}=1$ and $\max_Mh_i>0$. Since $\mathcal{H}$ is not empty, we can choose $(v_1,v_2)\in\mathcal{H}$, then
$$J_{4\pi-\ep,\rho_2}(u_1^\ep,u_2^\ep)=\inf_{(u_1,u_2)\in\mathcal{H}}J_{4\pi-\ep,\rho_2}(u_1,u_2)\leq J_{4\pi-\ep,\rho_2}(v_1,v_2)\to J_{4\pi,\rho_2}(v_1,v_2)\leq C.$$
This together with the Moser-Trudinger inequality \eqref{ineq-toda} and Jensen's inequality yields
\begin{align*}
\log\int_Me^{u_1^\ep}+\log\int_Me^{u_2^\ep}
\leq& \frac{1}{12\pi}\int_M(|\nabla u_1^\ep|^2+\nabla u_1^\ep\nabla u_2^\ep+|\nabla u_2^\ep|^2)+\overline{u_1^\ep}+\overline{u_2^\ep}+C\\
=& \frac{1}{4\pi}J_{4\pi-\ep,\rho_2}(u^\ep)+\frac{\ep}{4\pi}\overline{u_1^\ep}+\frac{4\pi-\rho_2}{4\pi}\overline{u_2^\ep}+C\\
\leq& \frac{\ep}{4\pi}\log\int_Me^{u_1^\ep}+\frac{4\pi-\rho_2}{4\pi}\log\int_Me^{u_2^\ep}+C.
\end{align*}
This combining with $\int_Me^{u_i^\ep}$ is bounded from below by some $C_1>0$ shows that $\int_M e^{u_i^\ep}\leq C_2$ for some $C_2>0$.
This completes the proof.
\end{proof}

\begin{lem}\label{lem-Ls}
For any $s\in(1,2)$, $\|\nabla u_i^\ep\|_{L^s(M)}\leq C$ for $i=1,2$.
\end{lem}
\begin{proof}
Let $s'=1/s>2$, we know by definition that
\begin{align*}
\|\nabla u_1^\ep\|_{L^s(M)}=\sup\left\{\left|\int_M\nabla u_1^\ep\nabla \phi\right|: \phi\in W^{1,s'}(M), \int_M \phi =0, \|\phi\|_{W^{1,s'}(M)}=1\right\}.
\end{align*}
The Sobolev embedding theorem shows that $\|\phi\|_{L^\infty(M)}\leq C$ for some constant $C$. Then it follows by Eq. \eqref{eq-uep} and Lemma \ref{lem-bound} that
\begin{align*}
\left|\int_M\nabla u_1^\ep\nabla \phi\right|
=&\left|\int_M\phi(-\Delta u_1^\ep)\right|\\
=&\left|\int_M \phi[(8\pi-2\ep)(h_1e^{u_1^\ep}-1)-\rho_2(h_2e^{u^\ep_2}-1)]\right|\\
\leq&C.
\end{align*}
Therefore we have $\|\nabla u_1^\ep\|_{L^s(M)}\leq C$. Similarly, we have $\|\nabla u_2^\ep\|_{L^s(M)}\leq C$. This ends the proof.
\end{proof}

Concerning $(u^\ep_1, u^\ep_2)$, we have the following equivalent characterizations.
\begin{lem}\label{lem-char} The following three items are equivalent:

$(i)$ $m_1^{\ep}+m_2^{\ep}\to+\infty$ as $\ep\to0$;

$(ii)$ $\int_M(|\nabla u^\ep_1|^2+\nabla u_1^\ep\nabla u_2^\ep+|\nabla u^\ep_2|^2)\to+\infty$ as $\ep\to0$;

$(iii)$ $\overline{u^\ep_1}+\overline{u^\ep_2}\to-\infty$ as $\ep\to0$.
\end{lem}

\begin{proof} $(ii)\Leftrightarrow(iii)$: Since $J_{4\pi-\ep,\rho_2}$ is bounded, $(ii)$ is equivalent to
\begin{align}\label{lem-char-1}
(4\pi-\ep)\overline{u_1^\ep}+\rho_2\overline{u_2^\ep}\to-\infty~~\text{as}~~\ep\to0.
\end{align}
Using Lemma \ref{lem-bound} and Jensen's inequality, we have $\overline{u_i^\ep}\leq C$ for $i=1,2$. Therefore, \eqref{lem-char-1}
is equivalent to $(iii)$ and then $(ii)$ is equivalent to $(iii)$.

$(i)\Rightarrow (ii)$: Suppose not, we have
\begin{align*}
\int_{M}|\nabla u_1^\ep|^2+\int_{M}|\nabla u_2^\ep|^2\leq2\int_M(|\nabla u^\ep_1|^2+\nabla u_1^\ep\nabla u_2^\ep+|\nabla u^\ep_2|^2)\leq C.
\end{align*}
Meanwhile, by $(ii)\Leftrightarrow(iii)$ one knows $\overline{u_1^\ep}+\overline{u_2^\ep}\geq-C$. So $\overline{u_i^\ep}$ is bounded for $i=1,2$. By Poincar\'{e}'s inequality, we have for
$i=1,2$ that
\begin{align*}
\int_M  (u_i^\ep)^2-\overline{u_i^\ep}^2=\int_M(u_i^\ep-\overline{u_i^\ep})^2\leq C\int_M|\nabla u_i^\ep|^2\leq C.
\end{align*}
So $(u_i^\ep)$ is bounded in $L^2(M)$. Since $\|\nabla u_i^\ep\|_{L^2(M)}$ and $\overline{u_i^\ep}$ are both bounded, we have by the Moser-Trudinger inequality that $(e^{u_i^\ep})$
is bounded in $L^s(M)$ for any $s\geq1$. Then by using elliptic estimates to \eqref{eq-uep} we obtain that $(u_i^\ep)$ is bounded in
$W^{2,2}(M)$ and then $\|u_i^\ep\|_{L^{\infty}(M)}$ is bounded. Therefore, $m_i^\ep\leq C$ for $i=1,2$. This contradicts $(i)$.

$(ii)\Rightarrow(i)$: If not, then we have $m_1^\ep+m_2^\ep\leq C$.
Using Lemma \ref{lem-bound}, we have $m_i^\ep\geq C$ for $i=1,2$. So $m_i^\ep$ is bounded for $i=1,2$.
Then $(e^{u_i^\ep})$ is bounded. Since by Lemma \ref{lem-Ls}, $u_i^\ep-\overline{u_i^\ep}$ is bounded in $L^s(M)$ for any $s>1$, we have by using
elliptic estimates to \eqref{eq-uep} that $u_i^\ep-\overline{u_i^\ep}$ is bounded. Since $(ii)\Leftrightarrow(iii)$, we have
$\overline{u_1^\ep}+\overline{u_2^\ep}\to-\infty$. Notice that $\overline{u_i^\ep}\leq C$, we have $\overline{u_1^\ep}$ or $\overline{u_2^\ep}$
tends to $-\infty$. Without loss of generality, suppose $\overline{u_1^\ep}$ tends to $-\infty$. Then
\begin{align*}
1=\int_Mh_1e^{u_1^\ep}=\int_Mh_1e^{u_1^\ep-\overline{u_1^\ep}}e^{\overline{u_1^\ep}}\to0~~\text{as}~~\ep\to0.
\end{align*}
This is a contradiction.
\end{proof}

\begin{definition}[Blow up]
We call $(u_1^\ep,u_2^\ep)$ blows up, if one of the three items in Lemma \ref{lem-char} holds.
\end{definition}

When $(u_1^\ep,u_2^\ep)$ blows up, there holds
\begin{lem}\label{lem-mean}
Let $(u_1^\ep,u_2^\ep)$ minimize $J_{4\pi-\ep,\rho_2}$ in $\mathcal{H}$. If $(u_1^\ep,u_2^\ep)$ blows up, then
$$\overline{u_1^{\ep}}\to-\infty~~\text{as}~~\ep\to0~~\text{and}~~\overline{u^\ep_2}\geq-C.$$
\end{lem}

\begin{proof}
Since $J_{4\pi-\ep,\rho_2}(u_1^\ep,u_2^\ep)$ is bounded, we have by \eqref{ineq-toda} that
\begin{align*}
C\geq& J_{4\pi-\ep,\rho_2}(u_1^\ep,u_2^\ep)\nonumber\\
 \geq& \frac{1}{3}\int_M(|\nabla u_1^\ep|^2+\nabla u_1^\ep\nabla u_2^\ep+|\nabla u_2^\ep|^2)+(4\pi-\ep)\overline{u_1^\ep}+\rho_2\overline{u_2^\ep}\nonumber\\
 \geq& C-\ep\overline{u_1^\ep}-(4\pi-\rho_2)\overline{u_2^\ep}.
\end{align*}
Since $\overline{u_1^\ep}\leq C$ and $\rho_2<4\pi$, we have
$$\overline{u_2^\ep}\geq-C.$$
If $(u_1^\ep,u_2^\ep)$ blows up, it follows from Lemma \ref{lem-char} that $\overline{u_1^\ep}\to-\infty$ as $\ep\to0$. This finishes the proof.
\end{proof}

If $(u_1^\ep,u_2^\ep)$ does not blow up, then by Lemma \ref{lem-char}, one can show that $(u_1^\ep,u_2^\ep)$ converges to $(u_1^0,u_2^0)$ in $\mathcal{H}$ and $(u_1^0,u_2^0)$ minimizes $J_{4\pi,\rho_2}$.
The proof of Theorem \ref{thm-zhu-1} terminates in this case. Therefore, we assume $(u_1^\ep,u_2^\ep)$ blows up in the rest of this paper.

By Lemma \ref{lem-Ls}, there exist $G_i$, $i=1,2$ such that $u_1^\ep-\overline{u_1^\ep}\rightharpoonup G_1$ and $u_2^\ep \rightharpoonup G_2$ weakly in $W^{1,s}(M)$ for any $1<s<2$ as $\ep\to0$.
Since $(e^{u_i^\ep})$ is bounded in $L^1(M)$ we may extract a subsequence (still denoted $e^{u_i^\ep}$) such that $e^{u_i^\ep}$ converges in the sense of measures on
$M$ to some nonnegative bounded measure $\mu_i$ for $i=1,2$. We set
\begin{align*}
\gamma_1=8\pi h_1\mu_1-\rho_2 h_2\mu_2,~~\gamma_2=2\rho_2 h_2\mu_2-4\pi h_1\mu_1
\end{align*}
and
\begin{align*}
S_i=\{x\in M: |\gamma_i(\{x\})|\geq4\pi\},~i=1,2.
\end{align*}
Let $S=S_1\cup S_2$. By Theorem 1 in \cite{BM}, it is easy to show that for any $\Omega\subset\subset M\setminus S$,
\begin{align}\label{outbound}
u_i^\ep-\overline{u_i^\ep}~\text{is~uniformly~bounded~in~} \Omega,~i=1,2.
\end{align}
Since $(u_1^\ep,u_2^\ep)$ blows up, we know $S$ is not empty (Or else, with a finite covering argument, we have by \eqref{outbound} that $\|u_i^\ep-\overline{u_i^\ep}\|_{L^\infty(M)}\leq C$, then $m^\ep_i\leq C$, this contradicts with Lemma \ref{lem-char} (i)). Meanwhile, by the definition of $S$, we have for any $x\in S$,
$$\mu_1(\{x\})\geq\frac{1}{4\max_M|h_1|}~~\text{or}~~\mu_2(\{x\})\geq\frac{\pi}{\rho_2\max_M|h_2|}.$$
In view of $\mu_1$ and $\mu_2$ are bounded, $S$ is a finite set. We denote $S=\{x_l\}_{l=1}^L$.
%Let $\beta_i=\lim\limits_{\ep\to0}e^{\overline{u_i^\ep}}$ for $i=1,2$, by Lemma \ref{lem-mean}, $\beta_1=0$ and $\beta_2>0$.%
It follows from
\eqref{outbound} and Fatou's lemma that
$$\mu_1=\sum_{l=1}^L\mu_1(\{x_l\})\delta_{x_l}~~\text{and}~~\mu_2=e^{G_2}+\sum_{l=1}^L\mu_2(\{x_l\})\delta_{x_l}.$$
\begin{lem}\label{lem-x1}
$\rm{supp}\mu_1$ is a single point set.
\end{lem}
\begin{proof}
It follows from Lemma \ref{lem-bound} that ${\rm{supp}}\mu_1\neq\emptyset$.
If there are two different points in ${\rm{supp}}\mu_1$, then by Lemma \ref{lem-bound} and the improved Moser-Trudinger inequality (cf. \cite{CL91b}, Theorem 2.1),
for any $\ep'>0$, there exists some $C=C(\ep')>0$ such that
\begin{align}\label{lem-blowupset-1}
C\leq\log\int_Me^{u_1^\ep}\leq(\frac{1}{32\pi}+\ep')\int_M|\nabla u_1^\ep|^2+\overline{u_1^\ep}+C.
\end{align}
Since
\begin{align}\label{lem-blowupset-2}
C\geq& J_{4\pi-\ep,\rho_2}(u_1^\ep,u_2^\ep)\nonumber\\
=&\frac{1}{3}\int_M(|\nabla u_1^\ep|^2+\nabla u_1^\ep\nabla u_2^\ep+|\nabla u_2^\ep|^2)+(4\pi-\ep)\overline{u_1^\ep}+\rho_2\overline{u_2^\ep}\nonumber\\
=&\frac{1}{4}\int_M|\nabla u_1^\ep|^2+(4\pi-\ep)\overline{u_1^\ep}+\frac{1}{3}\int_M|\nabla(u_2^\ep+\frac12u_1^\ep)|^2+\rho_2\overline{u_2^\ep}\nonumber\\
\geq&\frac{1}{4}\int_M|\nabla u_1^\ep|^2+(4\pi-\ep)\overline{u_1^\ep}-C,
\end{align}
then we have by combining \eqref{lem-blowupset-1} and \eqref{lem-blowupset-2} that
\begin{align*}
\overline{u_1^\ep}\geq-C.
\end{align*}
In view of Lemmas \ref{lem-char} and \ref{lem-mean}, this is a contradiction. Therefore, ${\rm{supp}}\mu_1$ is a single point set.
This completes the proof.
\end{proof}

Since by \eqref{outbound} we know ${\rm{supp}}\mu_1\subset S$, we can assume without loss of generality that
${\rm{supp}}\mu_1=\{x_1\}$. By noticing that $\int_Mh_1e^{u_1^\ep}=1$, we have $h_1\mu_1=\delta_{x_1}$.

\begin{lem}\label{lem-4pi} There holds
$\gamma_2(\{x_l\})\leq-4\pi$ if $x_l\neq x_1$ and $\gamma_2(\{x_1\})<4\pi$.
\end{lem}
\begin{proof}
Since of Lemma \ref{lem-mean} and \eqref{outbound}, we know $\overline{u_2^\ep}\geq-C$. For any $x_l\in S$, choosing $r>0$ sufficiently small, we have
$u_2^\ep\mid_{\partial B_r(x_l)}\geq-C_0$ for some constant $C_0$. Let $w_2^\ep$ be the solution of
\begin{align*}
\begin{cases}
-\Delta w_2^\ep = 2\rho_2(h_2e^{u^\ep_2}-1)-(4\pi-\ep)(h_1e^{u_1^\ep}-1)~~\text{in}~B_r(x_l),\\
w_2^\ep = -C_0~~\text{on}~\p B_r(x_l).
\end{cases}
\end{align*}
By the maximum principle $w_2^\ep\leq u_2^\ep$ in $B_r(x_l)$. Since $2\rho_2h_2e^{u^\ep_2}-(4\pi-\ep)h_1e^{u_1^\ep}$ is bounded in $L^1(B_r(x_l))$,
$w_2^\ep\rightharpoonup w_2$ weakly in $W^{1,s}(B_r(x_l))$ for any $1<s<2$, where $w_2$ is the solution of
\begin{align*}
\begin{cases}
-\Delta w_2 = 2\rho_2(h_2e^{G_2}-1)+4\pi+\gamma_2(\{x_l\})\delta_{x_l}~~\text{in}~B_r(x_l),\\
w_2 = -C_0~~\text{on}~\p B_r(x_l).
\end{cases}
\end{align*}
Since $h_1\mu_1=\delta_{x_1}$, if $\gamma_2(\{x_l\})>0$, then $h_2(x_l)>0$ and one has
\begin{align*}
 2\rho_2(h_2e^{G_2}-1)+4\pi \geq -C~~\text{near}~~x_l.
\end{align*}
Then we have $-\Delta w_2\geq\gamma_2(\{x_l\})\delta_{x_l}-C$ in $B_r(x_l)$ (Here, for simplicity, we assume $r$ is small enough to ensure $h_2(x_l)>0$
in $B_r(x_l)$). Therefore
$$w_2\geq-\frac{1}{2\pi}\gamma_2(\{x_l\})\log|x-x_l|-C~~\text{in}~B_r(x_l).$$
Thus $e^{w_2}\geq C/|x-x_l|^{\frac{\gamma_2(\{x_l\})}{2\pi}}$. Note that it follows by Fatou's lemma that
$$\int_{B_r(x_l)}e^{w_2}\leq\lim_{\ep\to0}\int_{B_r(x_l)}e^{w_2^\ep}\leq\lim_{\ep\to0}\int_{B_r(x_l)}e^{u_2^\ep}\leq C.$$
Then we have
$$\gamma_2(\{x_l\})<4\pi,~~\forall l=1,2,\cdots,L.$$
If $x_l\neq x_1$, we have $\gamma_2(\{x_l\})\leq-4\pi$. In fact, if $\gamma_2(\{x_l\})>-4\pi$, then since $x_l\neq x_1$, one has
\begin{align*}
\gamma_1(\{x_l\})=-\rho_2 h_2(x_l)\mu_2(\{x_l\})=-\frac{1}{2}\gamma_2(\{x_l\})\in(-2\pi,2\pi).
\end{align*}
Then $x_l\notin S$. A contradiction. This ends the proof.
\end{proof}

Now we are prepared to prove the following lemma, which can be seen as a key in the proof of our main theorem. We remark that this lemma is obtained much more directly
with the help of Proposition 2.4 in \cite{JLW06} when the prescribed functions $h_1$ and $h_2$ are positive. However, when $h_1$ or $h_2$ changes signs,
we do not have the counterpart of Proposition 2.4 in \cite{JLW06} in the hand, and therefore more effort is needed in our situation.

\begin{lem} We have $u_2^\ep\leq C$.
\end{lem}
\begin{proof}
By Lemma \ref{lem-4pi}, we divide the whole proof into two cases.

\textbf{Case 1} $\gamma_2(\{x_l\})\leq-4\pi$ ($x_l\neq x_1$).

In this case, we have $2\rho_2h_2(x_l)\mu_2(\{x_l\})\leq-4\pi$, then $h_2(x_l)<0$ and $\mu_2(\{x_l\})>0$. We can choose $r$ sufficiently small such that
$h_2(x)<0$ in $B_{r}(x_l)$. Consider
\begin{align*}
\begin{cases}
-\Delta v_1^\ep = -(4\pi-\ep)h_1e^{u_1^\ep}~~\text{in}~~B_{r}(x_l),\\
v_1^\ep=0~~\text{on}~~\partial B_{r}(x_l).
\end{cases}
\end{align*}
We define $v_2^\ep=u_2^\ep-\overline{u_2^\ep}-v_1^\ep$. Then $-\Delta v_2^\ep=-2\rho_2+(4\pi-\ep)+2\rho_2h_2e^{u_2^\ep}\leq-2\rho_2+(4\pi-\ep)$.
By Theorem 8.17 in \cite{GT} (or Theorem 4.1 in \cite{HL}) and Lemma \ref{lem-Ls}, we have
\begin{align*}
\sup_{B_{r/2}(x_l)}v_2^\ep \leq& C\left(\|(v_2^\ep)^+\|_{L^s(B_r(x_l))}+C\right)\nonumber\\
\leq& C\left(\|u_2^\ep-\overline{u_2^\ep}\|_{L^s(M)}+\|v_1^\ep\|_{L^s(B_r(x_l))}+C\right)\nonumber\\
\leq& C\left(\|\nabla u_2^\ep\|_{L^s(M)}+\|v_1^\ep\|_{L^s(B_r(x_l))}+C\right)\nonumber\\
\leq& C\left(\|v_1^\ep\|_{L^s(B_r(x_l))}+C\right).
\end{align*}
Notice that $h_1\mu_1=\delta_{x_1}$ and $x_l\neq x_1$, one has $\int_{B_r(x_l)}|h_1|e^{u_1^\ep}\to0$ as $\ep\to0$ for sufficiently small $r$.
It then follows from Theorem 1 in \cite{BM} that $\int_{B_r(x_l)}e^{t|v_1^\ep|}\leq C$ for some $t>1$, which yields that
$$\|v_1^\ep\|_{L^s(B_r(x_l))}\leq C.$$
Then we have
$$\sup_{B_{r/2}(x_l)}v_2^\ep\leq C.$$
Note that
\begin{align*}
\int_{B_{r/2}(x_l)}e^{su_2^\ep}=&\int_{B_{r/2}(x_l)}e^{s\overline{u_2^\ep}}e^{sv_2^\ep}e^{sv_1^\ep}\\
\leq& C\int_{B_{r/2}(x_l)}e^{s|v_1^\ep|}\\
\leq& C.
\end{align*}
Therefore, one has by H\"older's inequality that
\begin{align*}
\mu_2(\{x_l\})=\lim_{r\to0}\lim_{\ep\to0}\int_{B_{r/2}(x_l)}e^{u_2^\ep}
\leq\lim_{r\to0}\lim_{\ep\to0}\left(\int_{B_{r/2}(x_l)}e^{su_2^\ep}\right)^{1/s}\left(\text{vol}B_{r/2}(x_l)\right)^{1-1/s}=0,
\end{align*}
this is a contradiction with $\mu_2(\{x_l\})>0$. Hence, we obtain $S=\{x_1\}$.

\textbf{Case 2} $\gamma_2(\{x_1\})<4\pi$ (We shall divide this case into three subcases.)

\textbf{Case 2.1} $h_2(x_1)\mu_2{(x_1)}=0$.

Choosing $r>0$ sufficiently small such that $h_1(x)>0$ in $B_r(x_1)$. Let $z_1^\ep$ be the solution of
\begin{align*}
\begin{cases}
-\Delta z_1^\ep = 2\rho_2h_2e^{u_2^\ep}~~\text{in}~B_r(x_1),\\
z_1^\ep=0~~\text{on}~\p B_r(x_1).
\end{cases}
\end{align*}
Let $z_2^\ep=u_2^\ep-\overline{u_2^\ep}-z_1^\ep$ so that $-\Delta z_2^\ep \leq -2\rho_2+(4\pi-\ep)$.
By Theorem 8.17 in \cite{GT} (or Theorem 4.1 in \cite{HL}) and Lemma \ref{lem-Ls}, we have
\begin{align*}
\sup_{B_{r/2}(x_1)}z_2^\ep \leq& C\left(\|(z_2^\ep)^+\|_{L^s(B_r(x_1))}+C\right)\\
\leq& C\left(\|u_2^\ep-\overline{u_2^\ep}\|_{L^s(M)}+\|z_1^\ep\|_{L^s(B_r(x_1)))}+C\right)\\
\leq& C\left(\|\nabla u_2^\ep\|_{L^s(M)}+\|z_1^\ep\|_{L^s(B_r(x_1))}+C\right)\\
\leq& C\left(\|z_1^\ep\|_{L^s(B_r(x_1))}+C\right).
\end{align*}
Since $\int_{B_r(x_1)}|h_2|e^{u_2^\ep}\to0$ as $\ep\to0$ for sufficiently small $r$, it follows from Theorem 1 in \cite{BM} that $\int_{B_r(x_1)}e^{t|z_1^\ep|}\leq C$ for some $t>1$, which yields that
$$\|z_1^\ep\|_{L^s(B_r(x_1))}\leq C.$$
Then we have
$$\sup_{B_{r/2}(x_1)}z_2^\ep\leq C.$$
Note that
\begin{align*}
\int_{B_{r/2}(x_1)}e^{tu_2^\ep}=&\int_{B_{r/2}(x_1)}e^{t\overline{u_2^\ep}}e^{tz_2^\ep}e^{tz_1^\ep}\\
\leq& C\int_{B_{r/2}(x_1)}e^{t|z_1^\ep|}\\
\leq& C.
\end{align*}
By the standard elliptic estimates, we have
\begin{align*}
\|z_1^\ep\|_{L^{\infty}(B_{r/4}(x_1))} \leq C.
\end{align*}
Therefore, we obtain that
\begin{align*}%\label{eq-prop2}
u_2^\ep-\overline{u_2^\ep}\leq C~~\text{in}~~B_{r/4}(x_1).
\end{align*}

\textbf{Case 2.2} $h_2(x_1)\mu_2{(x_1)}>0$.

Consider the equation
\begin{align*}
\begin{cases}
-\Delta v_1^\ep = 2\rho_2h_2e^{u_2^\ep}-(4\pi-\ep)h_1e^{u_1^\ep}:=f_\ep~~~~\text{in}~~B_\delta(x_1),\\
v_1^\ep = 0~~\text{on}~~\partial B_\delta(x_1).
\end{cases}
\end{align*}
Define $v_2^\ep=u_2^\ep-\overline{u_2^\ep}-v_1^\ep$, then $-\Delta v_2^\ep = (4\pi-\ep)-2\rho_2$ in $B_\delta(x_1)$.
By Theorem 4.1 in \cite{HL} and Lemma \ref{lem-Ls}, we have
\begin{align*}
\sup_{B_{\delta/2}(x_1)}|v_2^\ep| \leq& C\left(\|v_2^\ep\|_{L^1(B_\delta(x_1))}+C\right)\\
\leq& C\left(\|u_2^\ep-\overline{u_2^\ep}\|_{L^1(M)}+\|v_1^\ep\|_{L^1(B_{\delta}(x_1))}+C\right)\\
\leq& C\left(\|\nabla u_2^\ep\|_{L^s(M)}+\|v_1^\ep\|_{L^1(B_{\delta}(x_1))}+C\right)\\
\leq& C\left(\|v_1^\ep\|_{L^1(B_{\delta}(x_1))}+C\right).
\end{align*}
Since in this case $\|f_\ep\|_{L^1(B_{\delta}(x_1))}<4\pi$ for sufficiently small $\ep>0$,
it then follows from Theorem 1 in \cite{BM} that $e^{s|v_1^\ep|}$ is bounded in $B_{\delta}(x_1)$ for some $s>1$, which yields that
$$\|v_1^\ep\|_{L^1(B_\delta(x_1))}\leq C.$$
Then we have
$$\sup_{B_{\delta/2}(x_1)}v_2^\ep\leq C.$$
Note that
\begin{align}\label{eq-prop0}
\int_{B_{\delta/2}(x_1)}e^{su_2^\ep}=&\int_{B_{\delta/2}(x_1)}e^{s\overline{u_2^\ep}}e^{sv_2^\ep}e^{sv_1^\ep}\nonumber\\
\leq& C\int_{B_{\delta/2}(x_1)}e^{s|v_1^\ep|}\nonumber\\
\leq& C.
\end{align}
Therefore, one has by H\"older's inequality that
\begin{align*}
\mu_2(\{x_1\})=\lim_{\delta\to0}\lim_{\ep\to0}\int_{B_{r/2}(x_l)}e^{u_2^\ep}
\leq\lim_{\delta\to0}\lim_{\ep\to0}\left(\int_{B_{r/2}(x_l)}e^{su_2^\ep}\right)^{1/s}\left(\text{vol}B_{r/2}(x_l)\right)^{1-1/s}=0,
\end{align*}
this is a contradiction with $\mu_2(\{x_1\})>0$. This shows that this subcase will not happen.

\textbf{Case 2.3} $h_2(x_1)\mu_2{(x_1)}<0$.

Since $S=\{x_1\}$, it follows by \eqref{outbound} that $u_2^\ep$ is locally uniformly bounded in $M\setminus\{x_1\}$. But in this subcase, we have $\mu_2(\{x_1\})>0$, then
$\max_{B_r(x_1)}u_2^\ep=\max_Mu_2^\ep\to+\infty$ as $\ep\to0$. We assume $u_2^\ep(x_2^\ep)=\max_{B_r(x_1)}u_2^\ep$, it is obvious that $x_2^\ep\to x_1$ as $\ep\to0$. At the
maximum point $x_2^\ep$, we have
\begin{align*}
-\Delta u_2^\ep(x_2^\ep) = 2\rho_2(h_2(x_2^\ep)e^{u_2^\ep(x_2^\ep)}-1)-(4\pi-\ep)(h_1(x_2^\ep)e^{u_1^\ep(x_2^\ep)}-1)<0.
\end{align*}
This is a contradiction. Therefore, this subcase will not happen either.

Concluding all the cases above, we finish the proof.
\end{proof}

Since $S=\{x_1\}$ and $h_1\mu_1=\delta_{x_1}$, we have $x_1^\ep\to x_1$ as $\ep\to0$ by \eqref{outbound}.
Let $(\Omega;(x^1,x^2))$ be an isothermal coordinate system around $x_1$ and we assume the metric to be
$$g|_\Omega = e^{\phi}((dx^1)^2+(dx^2)^2),~~\phi(0)=0.$$
We have
\begin{align}\label{bubble-0}
u^\ep_1(x^\ep_1+r_1^{\ep}x)-m_1^\ep\to -2\log(1+\pi h_1(x_1)|x|^2),
\end{align}
where $r_1^\ep=e^{-m_1^\ep/2}$. Recalling that for any $s\in(1,2)$, we have
$u^\ep_1-\overline{u^\ep_1}\to G_1$ weakly in $W^{1,s}(M)$ and strongly in $C^2_{\text{loc}}(M\setminus\{x_1\})$,
$u_2^\ep\to G_2$ weakly in $W^{1,s}(M)$ and strongly in $C^2_{\text{loc}}(M\setminus\{x_1\})$, where $G_1$ and $G_2$ satisfy
\begin{align}\label{eq-green-0}
\begin{cases}
-\Delta G_1 = 8\pi(\delta_{x_1}-1) - \rho_2(h_2e^{G_2}-1),\\
-\Delta G_2 = 2\rho_2(h_2e^{G_2}-1)-4\pi(\delta_{x_1}-1),\\
\int_MG_1=0,~\int_Mh_2e^{G_2}=1,~\sup_MG_2\leq C.
\end{cases}
\end{align}
Here $\delta_{x_1}$ is the Dirac distribution.

It was proved in \cite{LL05} (page 708) that
\begin{align}\label{green-expression}
G_1(x,x_1)=-4\log r+A_1(x_1)+f,~~G_2(x,x_1)=2\log r+A_2(x_1)+g,
\end{align}
where $r^2=x_1^2+x_2^2$, $A_i(x_1)$ $(i=1,2)$ are constants and $f,g$ are two smooth functions which are zero at $x_1$.

\section{ The lower bound for $J_{4\pi,\rho_2}$}
In this section, we shall give the first step in proving Theorem \ref{thm-zhu-1}: deriving an explicit lower bound of $J_{4\pi,\rho_2}$ when $(u_1^\ep,u_2^\ep)$ blows up.
%Comparing with the case $h_1$ is positive, we need exclude the blowup happens on non-positive points of $h_1$.

Define $v_2^\ep=\frac13(2u_2^\ep+u_1^\ep)-\frac13(2\overline{u_2^\ep}+\overline{u_1^\ep})$, we have
\begin{align*}
\begin{cases}
&-\Delta v_2^\ep = (4\pi-\ep)(h_2e^{u_2^\ep}-1),\\
&\int_Mv_2^\ep=0.
\end{cases}
\end{align*}
Notice that $u_2^\ep\leq C$, it follows from the standard elliptic estimates that $\|v_2^\ep\|_{C^1(M)}\leq C$. Then we obtain that
\begin{align*}
\frac{1}{3}\int_{B_{\delta}(x_1^\ep)}(|\nabla u_1^\ep|^2+\nabla u_1^\ep\nabla u_2^\ep+|\nabla u_2^\ep|^2)
=&\frac{1}{4}\int_{B_{\delta}(x_1^\ep)}|\nabla u_1^\ep|^2+\frac{3}{4}\int_{B_{\delta}(x_1^\ep)}|\nabla v_2^\ep|^2\\
=&\frac{1}{4}\int_{B_{\delta}(x_1^\ep)}|\nabla u_1^\ep|^2+O(\delta^2).
\end{align*}
Denote $w(x)=-2\log(1+\pi h_1(x_1)|x|^2)$, we have by \eqref{bubble-0} that
\begin{align*}
\frac{1}{4}\int_{B_{\delta}(x_1^\ep)}|\nabla u_1^\ep|^2=&\frac{1}{4}\int_{B_{L}}|\nabla w|^2\\
&+\frac{1}{4}\int_{B_{\delta}(x_1^\ep)\setminus B_{Lr_1^\ep}(x_1^\ep)}|\nabla u_1^\ep|^2+o_\ep(1)+O(\delta^2).
\end{align*}
To estimate $\int_{B_{\delta}(x_1^\ep)\setminus B_{Lr_1^\ep}(x_1^\ep)}|\nabla u_1^\ep|^2$, we shall follow \cite{LL05} closely. Let
\begin{align*}
a_1^\ep = \inf_{\partial B_{Lr_1^\ep}(x_1^\ep)}u_1^\ep,~~~~b_1^\ep = \sup_{\partial B_{Lr_1^\ep}(x_1^\ep)}u_1^\ep.
\end{align*}
We set $a_1^\ep-b_1^\ep=m_1^\ep-\overline{u_1^\ep}+d_1^\ep$. Then
\begin{align*}
d_1^\ep = w(L)-\sup_{\partial B_{\delta}(p)}G_1+o_{\ep}(1).
\end{align*}
Define $f_1^\ep=\max\{\min\{u_1^\ep,a_1^\ep\},b_1^\ep\}$. We have
\begin{align*}
\int_{B_{\delta}(x_1^\ep)\setminus B_{Lr_1^\ep}(x_1^\ep)}|\nabla u_1^\ep|^2
\geq&\int_{B_{\delta}(x_1^\ep)\setminus B_{Lr_1^\ep}(x_1^\ep)}|\nabla f_1^\ep|^2\\
=&\int_{B_{\delta}(x_1^\ep)\setminus B_{Lr_1^\ep}(x_1^\ep)}|\nabla_{\mathbb{R}^2} f_1^\ep|^2\\
\geq&\inf_{\Psi|_{\partial B_{Lr_1^\ep}(0)}=a_1^\ep,\Psi|_{\partial B_{\delta}(0)}=b_1^\ep}\int_{B_{\delta}(0)\setminus B_{Lr_1^\ep}(0)}|\nabla_{\mathbb{R}^2} \Psi|^2.
\end{align*}
By the Dirichlet's principle, we know
$$\inf_{\Psi|_{\partial B_{Lr_1^\ep}(0)}=a_1^\ep,\Psi|_{\partial B_{\delta}(0)}=b_1^\ep}\int_{B_{\delta}(0)\setminus B_{Lr_1^\ep}(0)}|\nabla_{\mathbb{R}^2} \Psi|^2$$
is uniquely attained by the following harmonic function
\begin{align*}
\begin{cases}
&-\Delta_{\mathbb{R}^2}\phi = 0,\\
&\phi|_{\partial B_{Lr_1^\ep}(0)}=a_1^\ep,\phi|_{\partial B_{\delta}(0)}=b_1^\ep.
\end{cases}
\end{align*}
Thus,
\begin{align*}
\phi=\frac{a_1^\ep-b_1^\ep}{-\log Lr_1^\ep+\log\delta}\log r - \frac{a_1^\ep\log\delta-b_1^\ep\log Lr_1^\ep}{-\log Lr_1^\ep+\log\delta},
\end{align*}
and then
\begin{align*}
\int_{B_{\delta}(0)\setminus B_{Lr_1^\ep}(0)}|\nabla_{\mathbb{R}^2}\phi|^2 = \frac{4\pi(a_1^\ep-b_1^\ep)^2}{-\log(Lr_1^\ep)^2+\log\delta^2}.
\end{align*}
Concluding, we have
\begin{align*}
\int_{B_{\delta}(x_1^\ep)\setminus B_{Lr_1^\ep}(x_1^\ep)}|\nabla u_1^\ep|^2
\geq \frac{4\pi(a_1^\ep-b_1^\ep)^2}{-\log(Lr_1^\ep)^2+\log\delta^2}.
\end{align*}
Since $-\log(r_1^\ep)^2=m_1^\ep$, we obtain
\begin{align}\label{eq-neck-1}
\int_{B_{\delta}(x_1^\ep)\setminus B_{Lr_1^\ep}(x_1^\ep)}|\nabla u_1^\ep|^2
\geq& 4\pi\frac{(m_1^\ep-\overline{u_1^\ep}+d_1^\ep)^2}{m_1^\ep-\log L^2+\log\delta^2}.
\end{align}
By \eqref{lem-blowupset-2}, one has
\begin{align}\label{eq-neck-2}
\frac{1}{4}\int_{B_{\delta}(x_1^\ep)\setminus B_{Lr_1^\ep}(x_1^\ep)}|\nabla u_1^\ep|^2+(4\pi-\ep)\overline{u_1^\ep}
\leq \frac{1}{4}\int_{M}|\nabla u_1^\ep|^2+(4\pi-\ep)\overline{u_1^\ep}\leq C.
\end{align}
It follows form \eqref{eq-neck-1} and \eqref{eq-neck-2} that
\begin{align}\label{eq-neck-3}
\pi\frac{(m_1^\ep-\overline{u_1^\ep}+d_1^\ep)^2}{m_1^\ep-\log L^2+\log\delta^2}+(4\pi-\ep)\overline{u_1^\ep}\leq C.
\end{align}
Recalling that $\overline{u_1^\ep}\to-\infty$ and $m_1^\ep\to+\infty$, we get from \eqref{eq-neck-3}
\begin{align}\label{eq-neck-4}
\frac{\overline{u_1^\ep}}{m_1^\ep}=-1+o_{\ep}(1)
\end{align}
by dividing both sides by $m_1^\ep$ and letting $\ep$ tend to $0$. Taking \eqref{eq-neck-4} into \eqref{eq-neck-1}, we have
\begin{align*}
\int_{B_{\delta}(x_1^\ep)\setminus B_{Lr_1^\ep}(x_1^\ep)}|\nabla u_1^\ep|^2
\geq4\pi\frac{(m_1^\ep-\overline{u_1^\ep})^2}{m_1^\ep}+16\pi\left(d_1^\ep+\log L^2-\log\delta^2+o_{\ep}(1)\right).
\end{align*}
Then
\begin{align}\label{energy-in}
&\frac{1}{3}\int_{B_{\delta}(x_1^\ep)}(|\nabla u_1^\ep|^2+\nabla u_1^\ep\nabla u_2^\ep+|\nabla u_2^\ep|^2)+(4\pi-\ep)\overline{u_1^\ep}+\rho_2\overline{u_2^\ep}\nonumber\\
\geq&-4\pi-4\pi\log(\pi h_1(x_1))-4\pi A_1(x_1)+8\pi\log\delta\nonumber\\
 &+\rho_2\int_MG_2+o_{\ep}(1)+o_{L}(1)+o_{\delta}(1).
\end{align}

Using \eqref{eq-green-0} and \eqref{green-expression}, one has
\begin{align}\label{energy-out}
&\frac{1}{3}\int_{M\setminus B_{\delta}(x_1^\ep)}(|\nabla u_1^\ep|^2+\nabla u_1^\ep\nabla u_2^\ep+|\nabla u_2^\ep|^2)\nonumber\\
=&\frac{\rho_2}{2}\int_{M}G_2(h_2e^{G_2}-1)-8\pi\log\delta+2\pi A_1(x_1)+o_{\ep}(1)+o_{\delta}(1).
\end{align}
Combining \eqref{energy-in} and \eqref{energy-out}, we have
\begin{align*}
J_{4\pi-\ep,\rho_2}(u_1^\ep,u_2^\ep)\geq&-4\pi-4\pi\log(\pi h_1(x_1))-2\pi A_1(x_1)
 +\frac{\rho_2}{2}\int_{M}G_2(h_2e^{G_2}+1)\nonumber\\
 &+o_{\ep}(1)+o_{L}(1)+o_{\delta}(1).
\end{align*}
By letting $\ep\to0$ first, then $L\to+\infty$ and then $\delta\to0$, we obtain finally that
\begin{align}\label{lower-bound}
\inf_{u\in\mathcal{H}}J_{4\pi,\rho_2}(u)\geq&-4\pi-4\pi\log(\pi h_1(x_1))-2\pi A_1(x_1)
 +\frac{\rho_2}{2}\int_{M}G_2(h_2e^{G_2}+1)\nonumber\\
 \geq&-4\pi-4\pi\log\pi-2\pi\max_{x\in M_+}\left(2\log h_1(x)+A_1(x)\right)\nonumber\\
 &+\frac{\rho_2}{2}\int_{M}G_2(h_2e^{G_2}+1).
\end{align}

\section{Completion of the proof of Theorem \ref{thm-zhu-1}}
In this section, we first outline the rest proof, then construct the blowup sequences like in \cite{LL05} and present our calculations.
\subsection{Outline of the rest proof}
Let $\phi_1^\ep$ and $\phi_2^\ep$ be defined as \cite{LL05} (see section 6). If the condition \eqref{zhu-cond-1} is satisfied on $M_+$, we can follow \cite{LL05} step by step to show that for sufficiently small $\ep$
\begin{align*}%\label{up-bound}
J_{4\pi,\rho_2}(\phi_1^\ep,\phi_2^\ep)<&-4\pi-4\pi\log\pi-2\pi\max_{x\in M_+}\left(2\log h_1(x)+A_1(x)\right)\nonumber\\
 &+\frac{\rho_2}{2}\int_{M}G_2(h_2e^{G_2}+1).
\end{align*}
It is easy to check that $\int_{M}h_1e^{\phi_1^\ep}>0$ and $\int_{M}h_2e^{\phi_2^\ep}>0$, we define
\begin{align*}
\widetilde{\phi_i^\ep}=\phi_i^\ep-\log\int_{M}h_ie^{\phi_i^\ep},~~i=1,2.
\end{align*}
Then $(\phi_1^\ep,\phi_2^\ep)\in\mathcal{H}$. Since $J_{4\pi,\rho_2}(u_1+c_1,u_2+c_2)=J_{4\pi,\rho_2}(u_1,u_2)$ for any $c_1,c_2\in\mathbb{R}$, we have for sufficiently small $\ep$ that
\begin{align}\label{up-bound-last}
\inf_{u\in\mathcal{H}}J_{4\pi,\rho_2}(u)\leq& J_{4\pi,\rho_2}(\widetilde{\phi_1^\ep},\widetilde{\phi_2^\ep})=J_{4\pi,\rho_2}(\phi_1^\ep,\phi_2^\ep)\nonumber\\
<&-4\pi-4\pi\log\pi-2\pi\max_{x\in M_+}\left(2\log h_1(x)+A_1(x)\right)\nonumber\\
 &+\frac{\rho_2}{2}\int_{M}G_2(h_2e^{G_2}+1).
\end{align}
Combining \eqref{lower-bound} and \eqref{up-bound-last}, one knows that $(u_1^\ep,u_2^\ep)$ does not blow up. So $(u_1^\ep,u_2^\ep)$ converges to some $(u_1^0,u_2^0)$
which minimizes $J_{4\pi,\rho_2}$ in $\mathcal{H}$ and solves \eqref{eq-jlw-1}. The smooth of $u_1^0$ and $u_2^0$ follows from the standard elliptic estimates. Finally,
we complete the proof of Theorem \ref{thm-zhu-1}.

\subsection{Test function}
Suppose that $2\log h_1(p)+A_1(p)=\max_{x\in M_+}(2\log h_1(x)+A_1(x))$.
Let $(\Omega;(x^1,x^2))$ be an isothermal coordinate system around $p$ and we assume the metric to be
$$g|_\Omega = e^{\phi}((dx^1)^2+(dx^2)^2),$$
and
$$\phi=b_1(p)x^1+b_2(p)x^2+c_1(p)(x^1)^2+c_2(p)(x^2)^2+c_{12}(p)x^1x^2+O(r^3),$$
where $r(x^1,x^2)=\sqrt{(x^1)^2+(x^2)^2}$.
Moreover we assume near $p$ that
\begin{align*}
G_i=a_i\log r+A_i(p)+&\lambda_i(p)x^1+\nu_i(p)x^2+\alpha_i(p)(x^1)^2+\beta_i(p)(x^2)^2+\xi_i(p)x^1x^2\\
    &+\ell_i(x^1,x^2)+O(r^4), i=1,2,
\end{align*}
where $a_1=-4$, $a_2=2$.
It is well known that
\begin{align*}
K(p)=-(c_1(p)+c_2(p)),\\
|\nabla u|^2dV_g=|\nabla u|^2dx^1dx^2,
\end{align*}
and
\begin{align*}
\frac{\p u}{\p n}dS_g=\frac{\p u}{\p r}rd\theta.
\end{align*}
For $\alpha_i$ and $\beta_i$, we have the following lemma:
\begin{lem}\label{lem-alpha}
We have
\begin{align*}
\alpha_1(p)+\beta_1(p)=4\pi-\frac{\rho_2}{2},~~\alpha_2(p)+\beta_2(p)=\rho_2-2\pi.
\end{align*}
\end{lem}
\begin{proof}
We have near $p$ that
\begin{align*}
2\alpha_1(p)+2\beta_1(p)+O(r)=\Delta_{\mathbb{R}^2}G_1=e^{-\phi}[8\pi+\rho_2(h_2e^{G_2}-1)],\\
2\alpha_2(p)+2\beta_2(p)+O(r)=\Delta_{\mathbb{R}^2}G_2=e^{-\phi}[-2\rho_2(h_2e^{G_2}-1)-4\pi],
\end{align*}
then the lemma is proved since $e^{G_2}=O(r^2)$ near $p$.
\end{proof}

We choose as in \cite{LL05} that
\begin{align*}
\phi_1^\ep=
\begin{cases}
w(\frac{x}{\ep})+\lambda_1(p)r\cos\theta+\nu_1(p)\sin\theta, &x\in B_{L\ep}(p),\\
G_1-\eta H_1+4\log(L\ep)-2\log(1+\pi L^2)-A_1(p),  &x\in B_{2L\ep}(p)\setminus B_{L\ep}(p),\\
G_1+4\log(L\ep)-2\log(1+\pi L^2)-A_1(p),&\text{otherwise}
\end{cases}
\end{align*}
and
\begin{align*}
\phi_2^\ep=
\begin{cases}
-\frac{w(\frac{x}{\ep})+2\log(1+\pi L^2)}{2}+2\log (L\ep)\\
~~~~+\lambda_2(p)r\cos\theta+\nu_2(p)r\sin\theta+A_2(p), &x\in B_{L\ep}(p),\\
G_2-\eta H_2,  &x\in B_{2L\ep}(p)\setminus B_{L\ep}(p),\\
G_2,&\text{otherwise}.
\end{cases}
\end{align*}
Here,
$$H_i=G_i-a_i\log r-A_i(p)-\lambda_i(p)r\cos\theta-\nu_i(p)r\sin\theta,~~i=1,2$$
and $\eta$ is a cut-off function which equals $1$ in $B_{L\ep}(p)$, equals $0$ in $M\setminus B_{2L\ep}(p)$ and satisfies $|\nabla\eta|\leq\frac{C}{L\ep}$.

Using Lemma 5.2 in \cite{LL05} and Lemma \ref{lem-alpha}, we have
\begin{align*}
\int_{M}|\nabla\phi_1^\ep|^2=&\int_{B_{L\ep}(p)}|\nabla\phi_1^\ep|^2+\int_{M\setminus B_{L\ep}(p)}|\nabla G_1|^2\\
&-2\int_{M\setminus B_{L\ep}(p)}\nabla G_1\nabla(\eta H_1)+\int_{M\setminus B_{L\ep}(p)}|\nabla(\eta H_1)|^2\\
=&\int_{B_L(0)}|\nabla_{\mathbb{R}^2} w|^2+\pi(\lambda_1^2(p)+\nu_1^2(p))(L\ep)^2-8\pi(4\pi-\frac{\rho_2}{2})(L\ep)^2\\
 &+\int_{M\setminus B_{L\ep}(p)}|\nabla G_1|^2+O((L\ep)^4),
\end{align*}
\begin{align*}
\int_{M}|\nabla\phi_2^\ep|^2=&\int_{B_{L\ep}(p)}|\nabla\phi_2^\ep|^2+\int_{M\setminus B_{L\ep}(p)}|\nabla G_2|^2\\
&-2\int_{M\setminus B_{L\ep}(p)}\nabla G_2\nabla(\eta H_2)+\int_{M\setminus B_{L\ep}(p)}|\nabla(\eta H_2)|^2\\
=&\frac{1}{4}\int_{B_L(0)}|\nabla_{\mathbb{R}^2} w|^2+\pi(\lambda_2^2(p)+\nu_2^2(p))(L\ep)^2+4\pi(\rho_2-2\pi)(L\ep)^2\\
 &+\int_{M\setminus B_{L\ep}(p)}|\nabla G_2|^2+O((L\ep)^4)
\end{align*}
and
\begin{align*}
\int_{M}\nabla\phi_1^\ep\nabla\phi_2^\ep=&\int_{B_{L\ep}(p)}\nabla\phi_1^\ep\nabla\phi_2^\ep+\int_{M\setminus B_{L\ep}(p)}\nabla G_1\nabla G_2\\
&-\int_{M\setminus B_{L\ep}(p)}(\nabla G_1\nabla(\eta H_2)+\nabla G_2\nabla(\eta H_1))+\int_{M\setminus B_{L\ep}(p)}\nabla(\eta H_1)\nabla(\eta H_2)\\
=&-\frac{1}{2}\int_{B_L(0)}|\nabla_{\mathbb{R}^2} w|^2+\pi(\lambda_1(p)\lambda_2(p)+\nu_1(p)\nu_2(p))(L\ep)^2\\
 &-4\pi(\rho_2-2\pi)(L\ep)^2+2\pi(4\pi-\frac{\rho_2}{2})(L\ep)^2\\
 &+\int_{M\setminus B_{L\ep}(p)}\nabla G_1\nabla G_2+O((L\ep)^4).
\end{align*}
Noticing that
\begin{align*}
 &\int_{M\setminus B_{L\ep}(p)}\left(|\nabla G_1|^2+|\nabla G_2|^2+\nabla G_1\nabla G_2\right)\nonumber\\
=&\int_{M\setminus B_{L\ep}(p)}\left(|\nabla G_1|^2+|\nabla G_2|^2+\frac{\nabla G_1\nabla G_2+\nabla G_2\nabla G_1}{2}\right)\nonumber\\
%=&\int_{M\setminus B_{L\ep}(p)}G_1(-\Delta G_1-\frac{1}{2}\Delta G_2)+\int_{M\setminus B_{L\ep}(p)}G_2(-\Delta G_2-\frac{1}{2}\Delta G_1)\nonumber\\
% &-\int_{\p B_{L\ep}(p)}\left(G_1\frac{\p G_1}{\p n}+G_2\frac{\p G_2}{\p n}+\frac{G_1\frac{\p G_2}{\p n}+G_2\frac{\p G_1}{\p n}}{2}\right)\nonumber\\
=&6\pi\int_{B_{L\ep}(p)}G_1+\frac{3}{2}\rho_2\int_{M}G_2(h_2e^{G_2}-1)+\frac{3}{2}\rho_2\int_{B_{L\ep}(p)}G_2\nonumber\\
 &-\int_{\p B_{L\ep}(p)}\left(G_1\frac{\p G_1}{\p n}+G_2\frac{\p G_2}{\p n}+\frac{G_1\frac{\p G_2}{\p n}+G_2\frac{\p G_1}{\p n}}{2}\right)\nonumber\\
 &+O((L\ep)^4\log(L\ep)).
\end{align*}
Calculating directly, we have
\begin{align*}
\int_{B_{L\ep}(p)}G_1 = -4\pi(L\ep)^2\log(L\ep)+2\pi(L\ep)^2+\pi A_1(p)(L\ep)^2+O((L\ep)^4\log(L\ep))
\end{align*}
and
\begin{align*}
\int_{B_{L\ep}(p)}G_2 = 2\pi(L\ep)^2\log(L\ep)-\pi(L\ep)^2+\pi A_2(p)(L\ep)^2+O((L\ep)^4\log(L\ep)).
\end{align*}
For the boundary terms, we use Lemma 5.2 in \cite{LL05} and Lemma \ref{lem-alpha} to calculate. Precisely, we have
\begin{align*}
\int_{\p B_{L\ep}(p)}G_1\frac{\p G_1}{\p n} =& 32\pi\log(L\ep)-4\pi(4\pi-\frac{\rho_2}{2})(L\ep)^2+\pi(\lambda^2_1(p)+\nu_1^2(p))(L\ep)^2\nonumber\\
                                              &-8\pi A_1(p)+2\pi(4\pi-\frac{\rho_2}{2})A_1(p)(L\ep)^2-8\pi(4\pi-\frac{\rho_2}{2})(L\ep)^2\log(L\ep)\nonumber\\
                                              &+O((L\ep)^4\log(L\ep)),\\
\int_{\p B_{L\ep}(p)}G_2\frac{\p G_2}{\p n} =& 8\pi\log(L\ep)+2\pi(\rho_2-2\pi)(L\ep)^2+\pi(\lambda^2_2(p)+\nu_2^2(p))(L\ep)^2\nonumber\\
                                              &+4\pi A_2(p)+4\pi(\rho_2-2\pi)A_2(p)(L\ep)^2+4\pi(\rho_2-2\pi)(L\ep)^2\log(L\ep)\nonumber\\
                                              &+O((L\ep)^4\log(L\ep)),\\
\int_{\p B_{L\ep}(p)}G_1\frac{\p G_2}{\p n} =& -16\pi\log(L\ep)-4\pi(\rho_2-2\pi)(L\ep)^2+\pi(\lambda_1(p)\lambda_2(p)+\nu_1(p)\nu_2(p))(L\ep)^2\nonumber\\
                                              &-8\pi A_2(p)+2\pi(4\pi-\frac{\rho_2}{2})A_2(p)(L\ep)^2+4\pi(\rho_2-2\pi)(L\ep)^2\log(L\ep)\nonumber\\
                                              &+O((L\ep)^4\log(L\ep)),\\
\end{align*}
\begin{align*}
\int_{\p B_{L\ep}(p)}G_2\frac{\p G_1}{\p n} =& -16\pi\log(L\ep)+2\pi(4\pi-\frac{\rho_2}{2})(L\ep)^2+\pi(\lambda_2(p)\lambda_1(p)+\nu_2(p)\nu_1(p))(L\ep)^2\nonumber\\
                                              &+4\pi A_1(p)+2\pi(\rho_2-2\pi)A_1(p)(L\ep)^2-8\pi(\rho_2-2\pi)(L\ep)^2\log(L\ep)\nonumber\\
                                              &+O((L\ep)^4\log(L\ep)).
\end{align*}
Therefore, we obtain that
\begin{align}\label{energy}
 &\frac{1}{3}\int_{M}\left(|\nabla\phi_1^\ep|^2+|\nabla\phi_2^\ep|^2+\nabla\phi_1^\ep\nabla\phi_2^\ep\right)\nonumber\\
=&4\pi\log(1+\pi L^2)-\frac{4\pi^2L^2}{1+\pi L^2}-8\pi\log(L\ep)+2\pi A_1(p)\nonumber\\
 &+\frac{1}{2}\rho_2\int_{M}G_2(h_2e^{G_2}-1)+O((L\ep)^4\log(L\ep)).
\end{align}

Do calculations, we have
\begin{align}\label{mean-1}
\int_{M}\phi_1^\ep=&\ep^2\int_{B_L(0)}we^{\phi(\ep x^1,\ep x^2)}+4\log(L\ep)+2\pi(L\ep)^2\log(1+\pi L^2)\nonumber\\
                   &-2\pi(L\ep)^2-A_1(p)-2\log(1+\pi L^2)+O((L\ep)^4\log(L\ep))
\end{align}
and
\begin{align}\label{mean-2}
\int_{M}\phi_2^\ep=&-\frac{\ep^2}{2}\int_{B_L(0)}we^{\phi(\ep x^1,\ep x^2)}-\pi(L\ep)^2\log(1+\pi L^2)\nonumber\\
                   &+\pi(L\ep)^2+\int_M G_2+O((L\ep)^4\log(L\ep)).
\end{align}
Since
\begin{align*}
\int_{B_L(0)}we^{\phi(\ep x^1,\ep x^2)}
=2\pi L^2-2\log(1+\pi L^2)-2\pi L^2\log(1+\pi L^2)+O(L^2\ep^2\log L),
\end{align*}
we obtain that by instituting this into \eqref{mean-1} and \eqref{mean-2} respectively
\begin{align}\label{mean-1-last}
\int_{M}\phi_1^\ep=&4\log(L\ep)-A_1(p)-2\log(1+\pi L^2)\nonumber\\
                   &-2\ep^2\log(1+\pi L^2)+O((L\ep)^4\log(L\ep))
\end{align}
and
\begin{align}\label{mean-2-last}
\int_{M}\phi_2^\ep=\ep^2\log(1+\pi L^2)+\int_M G_2+O((L\ep)^4\log(L\ep)).
\end{align}

Denoting $\mathcal{M}=\frac{1}{\pi}(-\frac{K(p)}{2}+\frac{(b_1(p)+\lambda_1(p))^2+(b_2(p)+\nu_1(p))^2}{4})$ and using $\alpha_1(p)+\beta_1(p)=4\pi-\frac{\rho_2}{2}$, we have
\begin{align}\label{ephi-in}
\int_{B_{L\ep}(p)}e^{\phi_1^\ep}=\ep^2\left(1-\frac{1}{1+\pi L^2}+\mathcal{M}\ep^2\log(1+\pi L^2)+O(\ep^2)+O(\ep^3\log L)\right),
\end{align}
\begin{align}\label{ephi-neck}
\int_{B_{\delta}(p)\setminus B_{L\ep}(p)}e^{\phi_1^\ep}=\ep^2\left(\frac{\pi L^2}{(1+\pi L^2)^2}-(\mathcal{M}+\frac{4\pi-\frac{\rho_2}{2}}{2\pi})\ep^2\log(L\ep)^2\right.\nonumber\\
\left.+O(\ep^2)+O(\frac{1}{L^4})\right),
\end{align}
and
\begin{align}\label{ephi-out}
\int_{M\setminus B_{\delta}(p)}e^{\phi_1^\ep}=O(\ep^4).
\end{align}
By combining \eqref{ephi-in}, \eqref{ephi-neck} and \eqref{ephi-out}, one has
\begin{align}\label{ephi}
\int_Me^{\phi_1^\ep} = \ep^2\left(1+\mathcal{M}\ep^2\log(1+\pi L^2)-(\mathcal{M}+\frac{4\pi-\frac{\rho_2}{2}}{2\pi})\ep^2\log(L\ep)^2\right.\nonumber\\
\left.+O(\ep^2)+O(\frac{1}{L^4})+O(\ep^3\log L)\right).
\end{align}

Suppose that in $B_\delta(p)$
\begin{align*}
h_1(x)-h_1(p)=&k_1r\cos\theta+k_2r\sin\theta\\
&+k_3r^2\cos^2\theta+2k_4\cos\theta\sin\theta+k_5r^2\sin^2\theta+O(r^3).
\end{align*}
It follows from a simple computation that
\begin{align}\label{hephi-in}
&\int_{B_{L\ep}(p)}(h_1-h_1(p))e^{\phi_1^\ep}\nonumber\\
 = &\frac{1}{2\pi}[k_3+k_5+k_1(b_1+\lambda_1)+k_2(b_2+\nu_1)]\ep^4\log(1+\pi L^2)+O(\ep^4),
\end{align}
\begin{align}\label{hephi-neck}
&\int_{B_{\delta}(p)\setminus B_{L\ep}(p)}(h_1-h_1(p))e^{\phi_1^\ep}\nonumber\\
 = &-\frac{1}{2\pi}[k_3+k_5+k_1(b_1+\lambda_1)+k_2(b_2+\nu_1)]\ep^4\log(L\ep)^2+O(\ep^4),
\end{align}
and
\begin{align}\label{hephi-out}
\int_{M\setminus B_{\delta}(p)}(h_1-h_1(p))e^{\phi_1^\ep}=O(\ep^4).
\end{align}
By \eqref{ephi}, \eqref{hephi-in}, \eqref{hephi-neck} and \eqref{hephi-out}, we know that
\begin{align*}
 &\int_Mh_1e^{\phi_1^\ep}=h_1(p)\int_Me^{\phi_1^\ep}+\int_M(h_1-h_1(p))e^{\phi_1^\ep}\nonumber\\
=&h_1(p)\ep^2\left(1+\mathcal{M}\ep^2\log(1+\pi L^2)-(\mathcal{M}+\frac{4\pi-\frac{\rho_2}{2}}{2\pi})\ep^2\log(L\ep)^2\right)\nonumber\\
 &+\frac{1}{2\pi}[k_3+k_5+k_1(b_1+\lambda_1)+k_2(b_2+\nu_1)]\ep^4\log(1+\pi L^2)\nonumber\\
 &-\frac{1}{2\pi}[k_3+k_5+k_1(b_1+\lambda_1)+k_2(b_2+\nu_1)]\ep^4\log(L\ep)^2\nonumber\\
 &+O(\ep^4)+O(\frac{\ep^2}{L^4})+O(\ep^5\log L).
\end{align*}
Then we have
\begin{align}\label{loghphi-1}
 &\log\int_Mh_1e^{\phi_1^\ep}\nonumber\\
=&\log h_1(p)+\log\ep^2\nonumber\\
&+\mathcal{M}\ep^2\log(1+\pi L^2)-(\mathcal{M}+\frac{4\pi-\frac{\rho_2}{2}}{2\pi})\ep^2\log(L\ep)^2\nonumber\\
 &+\frac{1}{2\pi h_1(p)}[k_3+k_5+k_1(b_1+\lambda_1)+k_2(b_2+\nu_1)]\ep^2\log(1+\pi L^2)\nonumber\\
 &-\frac{1}{2\pi h_1(p)}[k_3+k_5+k_1(b_1+\lambda_1)+k_2(b_2+\nu_1)]\ep^2\log(L\ep)^2\nonumber\\
 &+O(\ep^2)+O(\frac{1}{L^4}).
\end{align}
Direct calculation shows that
\begin{align*}
\int_{B_{2L\ep}(p)}e^{\phi_2^\ep}=O((L\ep)^4), ~~\int_{B_{2L\ep}(p)}e^{G_2}=O((L\ep)^4).
\end{align*}
Since $\int_{M}h_2e^{G_2}=1$, we obtain that
\begin{align}\label{loghphi-2}
\log\int_Mh_2e^{\phi_2^\ep}
=\log\left(1-O((L\ep)^4)\right)=O((L\ep)^4).
\end{align}
Taking \eqref{energy}, \eqref{mean-1-last}, \eqref{mean-2-last}, \eqref{loghphi-1} and \eqref{loghphi-2} into the functional, we obtain that
\begin{align*}
J_{4\pi,\rho_2}(\phi_1^\ep,\phi_2^\ep)
=&-4\pi-4\pi\log\pi-4\pi\log h_1(p)-2\pi A_1(p)+\frac{\rho_2}{2}\int_{M}G_2(h_2e^{G_2}+1)\nonumber\\
 &-4\pi[\mathcal{M}+\frac{4\pi-\frac{\rho_2}{2}}{2\pi}+\frac{k_3+k_5+k_1(b_1+\lambda_1)+k_2(b_2+\nu_1)}{2\pi h_1(p)}]\nonumber\\
 &~~~~\times\ep^2[\log(1+\pi L^2)-\log(L\ep)^2]\nonumber\\
 &+O(\ep^2)+O(\frac{1}{L^4})+O((L\ep)^4\log(L\ep))+O(\ep^3\log L).
\end{align*}
Note that under the assumption \eqref{zhu-cond-1}, we have
\begin{align*}
\mathcal{N}:=&\mathcal{M}+\frac{4\pi-\frac{\rho_2}{2}}{2\pi}+\frac{k_3+k_5+k_1(b_1+\lambda_1)+k_2(b_2+\nu_1)}{2\pi h_1(p)}\nonumber\\
=&-\frac{K(p)}{2\pi}+\frac{(b_1+\lambda_1)^2+(b_2+\mu_1)^2}{4\pi}+\frac{4\pi-\frac{\rho_2}{2}}{2\pi}+\frac{\frac{1}{2}\Delta h_1(p)+k_1(b_1+\lambda_1)+k_2(b_2+\nu_1)}{2\pi h_1(p)}\nonumber\\
=&\frac{1}{4\pi}\left[\Delta\log h_1(p)+8\pi-\rho_2-2K(p)\right]+\frac{1}{4\pi}\left[(b_1+\lambda_1+k_1)^2+(b_2+\nu_1+k_2)^2\right]\nonumber\\
>&0,
\end{align*}
where we have used $\Delta h_1(p)=\frac{1}{2}(k_3+k_5)$ and $\nabla h_1(p)=(k_1,k_2)$.

By choosing $L^4\ep^2=\frac{1}{\log(-\log\ep)}$, we have
\begin{align*}
 J_{4\pi,\rho_2}(\phi_1^\ep,\phi_2^\ep)
=&-4\pi-4\pi\log\pi-4\pi\log h_1(p)-2\pi A_1(p)+\frac{\rho_2}{2}\int_{M}G_2(h_2e^{G_2}+1)\nonumber\\
 &-4\pi\mathcal{N}\ep^2(-\log\ep^2)+o(\ep^2(-\log\ep^2)).
\end{align*}
Since $\mathcal{N}>0$, we have for sufficiently small $\ep$ that
\begin{align*}
 J_{4\pi,\rho_2}(\phi_1^\ep,\phi_2^\ep)
<-4\pi-4\pi\log\pi-4\pi\log h_1(p)-2\pi A_1(p)+\frac{\rho_2}{2}\int_{M}G_2(h_2e^{G_2}+1).
\end{align*}
This finishes the proof of Theorem \ref{thm-zhu-1}. $\hfill{\square}$

\vspace{1cm}

\noindent\textbf{Data Availability} Data sharing is not applicable to this article as obviously no datasets were generated or
analyzed during the current study.

\noindent\textbf{Conflict of interest} The authors have no Conflict of interest to declare that are relevant to the content of this
article.


\begin{thebibliography}{99}

\bibitem{BM} Brezis, Ha\"{\i}m and Merle, Frank, Uniform estimates and blow-up behavior of solutions of $-\Delta u=V(x)e^{u}$ in two dimensions,
   Comm. Partial Differential Equations 16 (1991), no. 8-9, 1223--1253.

\bibitem{CL91b} Chen, Wenxiong and  Li, Congming, Prescribing Gaussian curvatures on surfaces with conical singularities, J. Geom. Anal. 1 (1991), no. 4, 359--372.

\bibitem{DJLW97} Ding, Weiyue and Jost, J\"{u}rgen and Li, Jiayu and Wang, Guofang,
              The differential equation {$\Delta u=8\pi-8\pi he^u$} on a compact {R}iemann surface,
              Asian J. Math., 2 (1997), no. 2, 230--248.

\bibitem{Dun95} Dunne, Gerald, Self-dual Chern-Simons theories, Lecture Notes in Physics, 36. Springer, Berlin, 1995.

\bibitem{Fon} Fontana, Luigi, Sharp borderline {S}obolev inequalities on compact
              {R}iemannian manifolds, Comment. Math. Helv. 68 (1993), no.3, 415--454.

\bibitem{HL} Han, Qing and Lin, Fanghua, Elliptic partial differential equations, Courant Lecture Notes in Mathematics, Vol. 1,
          New York University, Courant Institute of Mathematical Sciences, New York; American Mathematical Society, Providence, RI, 1997.

\bibitem{GT} {Gilbarg, David and Trudinger, Neil S.},
      {Elliptic partial differential equations of second order},
    {Classics in Mathematics},
      {Reprint of the 1998 edition},
  {Springer-Verlag, Berlin},
       {2001}.

\bibitem{G97} {Guest, Martin A.}, {Harmonic maps, loop groups, and integrable systems}, {London Mathematical Society Student Texts},
    Vol. {38}, {Cambridge University Press, Cambridge}, {1997}.


\bibitem{JW01} Jost, J\"{u}rgen and Wang, Guofang,
               Analytic aspects of the {T}oda system. {I}. {A}
              {M}oser-{T}rudinger inequality,
               Comm. Pure Appl. Math., 54 (2001), no. 11, 1289--1319.

\bibitem{JLW06} Jost, J\"{u}rgen and Lin, Changshou and Wang, Guofang,
                Analytic aspects of the {T}oda system. {II}. {B}ubbling behavior and existence of solutions,
                Comm. Pure Appl. Math., 59 (2006), no. 4, 526--558.

\bibitem{KW74}  Kazdan, Jerry L. and Warner, F. W., Curvature functions for compact {$2$}-manifolds, Ann. of Math. (2) 99 (1974), 14--47.

\bibitem{LL05} Li, Jiayu and Li, Yuxiang, Solutions for {T}oda systems on {R}iemann surfaces,
              Ann. Sc. Norm. Super. Pisa Cl. Sci. (5), 4(2005), no. 4, 703--728.

\bibitem{LZ19} Li, Jiayu and Zhu Chaona, The convergence of the mean field type flow at a critical case, Calc. Var. Partial Differential Equations
         58 (2019), no. 2, Paper No. 60.

\bibitem{LX22} Li, Mingxiang and Xu, Xingwang, A flow approach to mean field equation, Calc. Var. Partial Differential Equations 61 (2022), no. 4, Paper No. 143.

\bibitem{Mart09} {Martinazzi, Luca},
{Concentration-compactness phenomena in the higher order
              {L}iouville's equation},
 {J. Funct. Anal.},
256 (2009), no. 11, {3743--3771},

\bibitem{SZ21} Sun, Linlin and Zhu, Jingyong, Global existence and convergence of a flow to {K}azdan-{W}arner equation with non-negative prescribed function,
      Calc. Var. Partial Differential Equations 60 (2021), no. 1, Paper No. 42.

\bibitem{SZ24} Sun, Linlin and Zhu, Jingyong, Existence of Kazdan-Warner equation with sign-changing prescribed function,
     {Calc. Var. Partial Differential Equations}, 63 (2024),  no. 2, Paper No. 52.

\bibitem{Tar08} Tarantello, Gabriella,
{Selfdual gauge field vortices}:  {an analytical approach},
{Progress in Nonlinear Differential Equations and their Applications}, 72,
{Birkh\"{a}user Boston, Inc., Boston, MA},
{2008}.

\bibitem{WY22} Wang, Yamin and Yang, Yunyan, A mean field type flow with sign-changing prescribed function on a symmetric {R}iemann surface,
             J. Funct. Anal. 282 (2022), no. 11, Paper No. 109449.

\bibitem{Ya01} Yang, Yisong,
{Solitons in field theory and nonlinear analysis},
{Springer Monographs in Mathematics},
{Springer-Verlag, New York},
{2001}.

\bibitem{YZ17} Yang, Yunyan and Zhu, Xiaobao, A remark on a result of {D}ing-{J}ost-{L}i-{W}ang,
              Proc. Amer. Math. Soc., 145 (2017), no. 9, 3953--3959.

\bibitem{YuZ24+} Yu, Pengxiu and Zhu, Xiaobao, Extremal functions for a Trudinger-Moser inequality with a sign-changing weight,
  Potential Anal. (2024).

\bibitem{Z24} Zhu, Xiaobao, Another remark on a result of Ding-Jost-Li-Wang,
             Proc. Amer. Math. Soc., 152 (2024), no. 2, 639--651.

%\bibitem{BM} Brezis, H.,  Merle, F., Uniform estimates and blow-up behavior of solutions of $-\Delta u=V(x)e^{u}$ in two dimensions,
%   Comm. Partial Differential Equations 16 (1991), no. 8-9, 1223-1253.
%
%\bibitem{Ca15} Cast\'{e}ras, J.-B., A mean field type flow {II}: {E}xistence and convergence, Pacific J. Math. 276 (2015), no. 2, 321-345.
%
%\bibitem{Chen1990} Chen, W., A Tr\"{u}dinger inequality on surfaces with conical sigularities, Proc. Amer. Math. Soc. 108 (1990), no. 3,  821-832.
%
%\bibitem{CL} Chen, W., Li, C., Classification of solutions of some nonlinear elliptic equations, Duke Math. J. 63 (1991), no. 3, 615-622.
%
%\bibitem{CL-91-JGA} Chen, W., Li, C., Prescribing Gaussian curvatures on surfaces with conical singularities, J. Geom. Anal. 1 (1991), no. 4, 359-372.
%
%\bibitem{CM} Colding, T., Minicozzi, II, W. P., A course in minimal surfaces, Graduate Studies in Mathematics Vol. 121 (2011), American Mathematical Society, Providence, RI.
%
%\bibitem{DJLW} Ding, W., Jost, J., Li, J., Wang, G., The differential equation $\Delta u= 8\pi-8\pi he^u$ on a compact Riemann surface,
%   Asian J. Math. 1 (1997), no. 2, 230-248.
%
%\bibitem{KW74}  Kazdan, J., Warner, F., Curvature functions for compact {$2$}-manifolds, Ann. of Math. (2) 99 (1974), 14-47.
%
%\bibitem{LZ19} Li, J., Zhu, C., The convergence of the mean field type flow at a critical case, Calc. Var. Partial Differential Equations
%         58 (2019), no. 2, Paper No. 60.
%
%\bibitem{LX2022} Li, M., Xu, X., A flow approach to mean field equation, Calc. Var. Partial Differential Equations 61 (2022), no. 4, Paper No. 143.
%
%\bibitem{Li01} Li, Y., Moser-Trudinger inequality on compact Riemannian manifolds of dimension two, J. Partial Differential Equations 14 (2001), no. 2, 163-192.
%
%\bibitem{Mancini} Mancini, G., Onofri-type inequalities for singular Liouville equations, J. Geom. Anal. 26 (2016), no. 2, 1202-1230.
%
%\bibitem{Struwe} Struwe, M., Variational method, vol. 34, Springer-Verlag, Berlin, 1996, xvi+272 pp. ISBN: 3-540-58859-0.
%
%\bibitem{SZ2021} Sun, L., Zhu, J., Global existence and convergence of a flow to {K}azdan-{W}arner equation with non-negative prescribed function,
%        Calc. Var. Partial Differential Equations 60 (2021), no. 1, Paper No. 42.
%
%\bibitem{SZ2022+} Sun, L., Zhu, J., Existence of Kazdan-Warner equation with sign-changing prescribed function, arXiv:2012.12840.
%
%\bibitem{Tro1991} Troyanov, M., Prescribing curvature on compact surfaces with conical singularities, Tran. Amer. Math. Soc. 324 (1991) no. 2, 793-821.
%
%\bibitem{WY2022} Wang, Y., Yang, Y., A mean field type flow with sign-changing prescribed function on a symmetric {R}iemann surface,
%               J. Funct. Anal. 282 (2022), no. 11, Paper No. 109449.
%
%\bibitem{Yang16} Yang, Y., A {T}rudinger-{M}oser inequality on a compact {R}iemannian
%              surface involving {G}aussian curvature, J. Geom. Anal. 26 (2016), no. 4, 2893-2913.
%
%\bibitem{YZ17} Yang, Y., Zhu, X., A remark on a result of {D}ing-{J}ost-{L}i-{W}ang, Proc. Amer. Math. Soc. 145 (2017), no. 9, 3953-3959.
%
%\bibitem{YZ2018} Yang, Y., Zhu, X., Existence of solutions to a class of {K}azdan-{W}arner equations on compact {R}iemannian surface,
%               Sci. China Math. 61 (2018), no. 6, 1109-1128.
%
%\bibitem{Zhu17} Zhu, X., A weak Trudinger-Moser inequality with a singular weight on a compact Riemannian surface, Commun. Math. Stat. 5 (2017), no. 1, 37-57.
%
%\bibitem{Zhu2018} Zhu, X., A generalized {T}rudinger-{M}oser inequality on a compact
%              {R}iemannian surface, Nonlinear Anal. 169 (2018), 38-58.
%
%\bibitem{Zhu2022+} Zhu, X., Another remark on a result of Ding-Jost-Li-Wang, arXiv:2212.09943.

\end{thebibliography}
\end{document}